\documentclass[10pt,a4paper,reqno]{amsart}
\usepackage[utf8]{inputenc}
\usepackage[T1]{fontenc}       
\usepackage{ae}
\usepackage[french,english]{babel}   
\usepackage{amsfonts}          
\usepackage{amsmath}
\usepackage{amsthm}
\usepackage{amssymb}
\usepackage{mathtools}
\usepackage{calc}
\usepackage{centernot}
\usepackage{color}
\usepackage[pdftex]{hyperref}
\usepackage{mathrsfs}
\usepackage[marginpar]{todo}

\newcommand{\ve}{\varepsilon}
\newcommand{\N}{\mathbb{N}}

\newcommand{\R}{\mathbb{R}}

\newcommand{\Rn}{\mathbb{R}^n}

\newcommand{\Ss}{\mathbb{S}}
\newcommand{\abs}[1]{\lvert #1 \rvert}

\newcommand{\ang}[1]{\langle #1 \rangle}

\newcommand{\loc}{{\rm loc}}

\newcommand{\QQ}{\mathcal{Q}}

\newcommand{\F}{\mathcal{F}}

\newcommand{\p}{\partial}

\newcommand{\dv}{\text{div }}

\newcommand{\inb}{\partial_{\infty}}
\newcommand{\an}{\sphericalangle}
\newcommand{\pinf}{\partial_{\infty}}
\newcommand{\bM}{\bar{M}}

\numberwithin{equation}{section}

\theoremstyle{plain}     
\newtheorem{thm}{Lause}[section]
\newtheorem{theo}[thm]{Theorem}
\newtheorem{lem}[thm]{Lemma}

\newtheorem{prop}[thm]{Proposition}
\newtheorem{rem}[thm]{Remark}

\theoremstyle{definition}

\title[Pinching condition]{Solvability of minimal graph equation under pointwise pinching 
condition for sectional curvatures}

\author{Jean-Baptiste Casteras}
\address{J.-B. Casteras, Departement de Mathematique
Universite libre de Bruxelles, CP 214, Boulevard du Triomphe, B-1050 Bruxelles, Belgium}
\email{jeanbaptiste.casteras@gmail.com}

\author{Esko Heinonen}
\address{E. Heinonen, Department of Mathematics and Statistics, P.O.B. 68 (Gustaf 
H\"alstr\"omin katu 2b), 00014 University
of Helsinki, Finland.}
\email{esko.heinonen@helsinki.fi}

\author{Ilkka Holopainen}
\address{I.Holopainen, Department of Mathematics and Statistics, P.O.B. 68 (Gustaf 
H\"alstr\"omin katu 2b), 00014 University
of Helsinki, Finland.}
\email{ilkka.holopainen@helsinki.fi}

\thanks{J.-B.C. supported by MIS F.4508.14 (FNRS)}
\thanks{E.H. supported by the Academy of Finland, project 252293 and the Wihuri Foundation.}
\thanks{I.H. supported by the Academy of Finland, project 252293.}

\keywords{minimal graph equation, Dirichlet problem, Hadamard manifold.}

\subjclass[2000]{Primary 58J32; Secondary 53C21}

\begin{document}

\begin{abstract}
 We study the asymptotic Dirichlet problem for the minimal graph equation on a Cartan-Hadamard
 manifold $M$ whose radial sectional curvatures outside a compact set satisfy an upper bound
    $$
      K(P)\le - \frac{\phi(\phi-1)}{r(x)^2}
    $$
 and a pointwise pinching condition
    $$
      \abs{K(P)}\le C_K\abs{K(P')}
    $$
 for some constants $\phi>1$ and $C_K\ge 1$, where $P$ and $P'$ are any 2-dimensional subspaces
 of $T_xM$ containing the (radial) vector $\nabla r(x)$ and $r(x)=d(o,x)$ is the distance to
 a fixed point $o\in M$. We solve the asymptotic Dirichlet problem with any continuous
 boundary data for dimensions $n=\dim M>4/\phi+1$.
\end{abstract}

\maketitle
 
 \section{Introduction}
\noindent In this paper we are interested in the asymptotic Dirichlet problem for minimal graph
 equation
  \begin{equation}\label{mingraph}
    \dv \frac{\nabla u}{\sqrt{1+\abs{\nabla u}^2}} = 0
  \end{equation}
on a Cartan-Hadamard manifold $M$ of dimension $n\ge2$. We recall that a Cartan-Hadamard
manifold is a simply connected complete Riemannian manifold with non-positive sectional 
curvature. Since the exponential map $\exp_o \colon T_oM \to M$ is a diffeomorphism for 
every point $o\in M$, it follows that $M$ is diffeomorphic to $\Rn$. One can define an 
asymptotic boundary $\pinf M$ of $M$ as the set of all equivalence classes of unit speed
geodesic rays on $M$. Then the compactification of $M$ is given by $\bM = M\cup \pinf M$
equipped with the cone topology. We also notice that $\bM$ is homeomorphic to the 
closed Euclidean unit ball; for details, see \cite{eberleinoneill}. 

The \emph{asymptotic Dirichlet problem} on $M$ for some
operator $\QQ$ is the following: Given a function $f \in C(\pinf M)$ does there exist a
(unique) function $u \in C(\bM)$ such that $\QQ[u]=0$ on $M$ and $u|\pinf M = f$?
We will consider this problem for the minimal graph operator (or the mean curvature operator)
appearing in \eqref{mingraph}. It is also worth noting that a function $u$ satisfies 
\eqref{mingraph} if and only if the graph $\{(x,u(x)) \colon x\in M\}$ is a minimal 
hypersurface in the product space $M\times\R$.
 
The asymptotic Dirichlet problem on Cartan-Hadamard manifolds has been solved for various
operators and under various assumptions on the manifold. The first result for this problem 
was due to Choi \cite{choi} when he solved the asymptotic Dirichlet problem for the Laplacian 
assuming that the sectional curvature has a negative upper bound $K_M \le -a^2 <0$, and that  
any two points at infinity can be separated by convex neighborhoods. Anderson \cite{anderson}
showed that such convex sets exist provided the sectional curvature of the manifold satisfies 
$-b^2\le K_M \le -a^2 <0$. We point out that Sullivan \cite{sullivan} solved independently the 
asymptotic Dirichlet problem for the Laplacian under the same curvature assumptions but
using probabilistic arguments. Cheng \cite{cheng} was the first to solve the problem for the 
Laplacian under the
same type of pointwise pinching assumption for the sectional curvatures as we consider in this paper.
Later the asymptotic Dirichlet problem has been generalized for $p$-harmonic and 
$\mathcal{A}$-harmonic functions under various curvature assumptions, see 
\cite{casterasholopainenripoll2}, \cite{holopainenplaplace}, \cite{holopainenvahakangas}, 
\cite{vahakangaspinch}, \cite{vahakangasyoung}.
 
 Concerning the mean curvature operator, there has been a growing interest in developing a 
 theory of constant (or prescribed) mean curvature hypersurfaces in Riemannian manifolds. 
 For instance, Guan and Spruck \cite{GS1} investigated the problem of finding complete 
 hypersurfaces of constant mean curvature with prescribed asymptotic boundaries at infinity 
 in the hyperbolic space (see also the recent \cite{GSX} and references therein). On 
 the other hand, Dajczer, Hinojosa, and de Lira (\cite{DHdL}, \cite{DdL1}, \cite{DdL2}) 
 have studied Killing graphs of prescribed mean curvature under curvature conditions on the 
 ambient space. Further studies include so-called half-space theorems in product spaces 
 $M\times\mathbb{R}_{+}$; see \cite{RSS}, \cite{DJX}, and references therein. In these 
 investigations, {\em a priori} gradient estimates based on the classical maximum principle 
 for elliptic equations are indispensable. To motivate further the study of the asymptotic 
 Dirichlet problem for the minimal graph equation, we recall the papers 
 \cite{collinrosenberg} and \cite{galvezrosenberg} by Collin, G{\'a}lvez, and Rosenberg
 who were able to construct harmonic diffeomorphisms from the complex plane $\mathbb{C}$ 
 onto the hyperbolic plane $\mathbb{H}^2$ and onto any Hadamard surface $M$ whose curvature 
 is bounded from above by a negative constant, respectively, hence disproving a conjecture 
 of Schoen and Yau \cite{schoenyau}. The key idea in their constructions was to solve the 
 Dirichlet problem on unbounded ideal polygons with boundary values $\pm\infty$ on the 
 sides of the ideal polygons.

 
Concerning the asymptotic Dirichlet problem for the equation \eqref{mingraph}, Casteras, 
Holopainen, and Ripoll studied the problem under curvature bounds
    $$
      -b\big(r(x)\big)^2 \le K(P) \le -a\big(r(x)\big)^2,
    $$
where $a,b \colon [0,\infty) \to [0,\infty)$ are smooth functions subject to some growth 
conditions. Here and throughout the paper $r(x)=d(x,o)$ stands for the distance to a fixed
point $o\in M$. As special cases of their main theorem 
\cite[Theorem 1.6]{casterasholopainenripoll1} we state here the following two solvability
results.
 \begin{theo}\cite[Theorem 1.5, Corollary 1.7]{casterasholopainenripoll1}
  Let $M$ be a Cartan-Hadamard manifold of dimension $n\ge2$. Suppose that 
      \begin{equation}
	-r(x)^{2(\phi-2)-\ve} \le K(P) \le -\frac{\phi(\phi-1)}{r(x)^2}
      \end{equation}
   or
      \begin{equation}\label{rttams}
	-r(x)^{-2-\ve}e^{2kr(x)} \le K(P) \le -k^2
      \end{equation}
   for some constants $\ve>0,\, \phi>1$, and $k>0$, and for all 2-dimensional subspaces 
   $P\subset T_xM$, with $x\in M\setminus B(o,R_0)$. Then the asymptotic Dirichlet problem
   for \eqref{mingraph} is uniquely solvable for any boundary data $f\in C(\pinf M)$.
 \end{theo}
The solvability of the asymptotic Dirichlet problem for \eqref{mingraph} under curvature 
assumptions \eqref{rttams}
was earlier obtained by Ripoll and Telichevesky in \cite{RT-tams}; see also \cite{E-SFR} 
and \cite{E-SR}. Recently, Casteras, Holopainen, and Ripoll \cite{casterasholopainenripoll2} 
were able to weaken the curvature upper bound to an almost optimal one.
  \begin{theo}\cite[Theorem 5]{casterasholopainenripoll2}
   Let $M$ be a Cartan-Hadamard manifold of dimension $n\ge3$ satisfying the curvature 
   assumption
      $$
	-\frac{\big(\log r(x)\big)^{2\bar{\ve}}}{r(x)^2} \le K(P) \le -\frac{1+\ve}{r(x)^2
	\log r(x)}
      $$
   for some constants $\ve>\bar{\ve}>0$ and for any 2-dimensional subspace $P\subset T_xM$,
   with $x\in M\setminus B(o,R_0)$. Then the asymptotic Dirichlet problem
   for \eqref{mingraph} is uniquely solvable for any boundary data $f\in C(\pinf M)$.
  \end{theo}
It is worth noting that even a strict negative curvature upper bound alone is not sufficient 
in dimensions $n\ge 3$ for the solvability of the asymptotic Dirichlet problem for 
\eqref{mingraph}. Indeed, in \cite{HR_ns} Holopainen and Ripoll generalized Borb\'ely's 
counterexample \cite{Bor} to cover the minimal graph equation.
 
Our main theorem is the following. It is worth noticing that no lower bounds for sectional 
curvatures are needed. Instead we assume a pointwise pinching condition on sectional curvatures.
 \begin{theo}\label{maintheorem}
  Let $M$ be a Cartan-Hadamard manifold of dimension $n\ge 2$ and let $\phi>1$. Assume that 
      \begin{equation}\label{curvassump}
       K(P) \le - \frac{\phi(\phi-1)}{r(x)^2},
      \end{equation}
where $K(P)$ is the sectional curvature of any two-dimensional subspace $P\subset T_xM$
containing the radial vector $\nabla r(x)$, with $x\in M\setminus B(o,R_0)$. Suppose also
that there exists a constant $C_K<\infty$ such that 
      \begin{equation}\label{pinchassump}
       \abs{K(P)} \le C_K\abs{K(P')}
      \end{equation}
whenever $x\in M\setminus B(o,R_0)$ and $P,P'\subset T_xM$ are two-dimensional subspaces 
containing the radial vector $\nabla r(x)$. Moreover, suppose that the dimension $n$ and
the constant $\phi$ satisfy the relation
      \begin{equation}\label{dimrestriction}
       n> \frac{4}{\phi} +1.
      \end{equation}
Then the asymptotic Dirichlet problem for the minimal graph equation \eqref{mingraph} is
uniquely solvable for any boundary data $f\in C(\pinf M)$.
 \end{theo} 
We notice that if we choose the constant $\phi$ in the curvature assumption to be bigger 
than $4$, then our theorem will hold in every dimension $n\ge2$. Similarly, if we let the 
dimension $n$ to be at least $5$, we can take the constant $\phi$ to be as close to $1$ as
we wish.

 In this paper we will proceed as follows. Section 2 is devoted to preliminaries. We will
recall some facts about Cartan-Hadamard manifolds, Jacobi equations, the minimal
graph equation and Young functions. In Section 3 we will prove our main theorem i.e. the 
solvability of the minimal graph equation under the curvature assumptions \eqref{curvassump},
\eqref{pinchassump} and \eqref{dimrestriction}. We will adopt
the strategies used in \cite{casterasholopainenripoll2}, \cite{cheng}, \cite{vahakangaspinch}
and \cite{vahakangasyoung}.

 \section{Preliminaries}
 
 \subsection{Cartan-Hadamard manifolds}
 Recall that a Cartan-Hadamard manifold is a complete and simply connected Riemannian
 manifold with non-positive sectional curvature. Let $M$ be a Cartan-Hadamard manifold and
 $\inb M$ the sphere at infinity, then we denote $\bM=M\cup\inb M$. The sphere at infinity
 is defined as the set of all equivalence classes of unit speed geodesic rays in $M$; two
 such rays $\gamma_1$ and $\gamma_2$ are equivalent if
    $$
      \sup_{t\ge0} d\big(\gamma_1(t),\gamma_2(t)\big) < \infty.
    $$
 The equivalence class of $\gamma$ is denoted by $\gamma(\infty)$. For each $x\in M$ and
 $y\in \bM\setminus\{x\}$ there exists a unique unit speed geodesic $\gamma^{x,y}\colon
 \R \to M$ such that $\gamma^{x,y}(0)=x$ and $\gamma^{x,y}(t)=y$ for some $t\in(0,\infty]$.
 For $x\in M$ and $y,z\in \bM\setminus\{x\}$ we denote by
    $$
      \an_x(y,z)=\an(\dot{\gamma}_0^{x,y},\dot{\gamma}_0^{x,z})
    $$
 the angle between vectors $\dot{\gamma}_0^{x,y}$ and $\dot{\gamma}_0^{x,z}$ in $T_xM$.
 If $v\in T_xM\setminus\{0\}$, $\alpha>0$, and $R>0$, we define a cone
    $$
      C(v,\alpha) = \{y\in \bM\setminus\{x\} \colon \an(v,\dot{\gamma}_0^{x,y})<\alpha\}
    $$
 and a truncated cone
    $$
      T(v,\alpha,R) = C(v,\alpha)\setminus \bar{B}(x,R).
    $$
 All cones and open balls in $M$ form a basis for the cone topology in $\bM$. With this
 topology $\bM$ is homeomorphic to the closed unit ball $\bar{B}^n\subset \Rn$ and
 $\inb M$ to the unit sphere $\Ss^{n-1} = \partial B^n$. For detailed study on the cone topology,
 see \cite{eberleinoneill}.

 Let us recall that the local Sobolev inequality holds on any Cartan-Hadamard manifold $M$.
 More precisely, there exist constants $r_S>0$ and $C_S<\infty$ such that
   \begin{equation}\label{sobieq}
    \left(\int_B \abs{\eta}^{n/(n-1)}\right)^{(n-1)/n} \le C_S \int_B \abs{\nabla \eta}
   \end{equation}
holds for every ball $B=B(x,r_S) \subset M$ and every function $\eta \in C_0^{\infty}(B)$.
This inequality can be obtained e.g. from Croke's estimate of the isoperimetric constant,
see \cite{cheegergromovtaylor} and \cite{croke}.
 
 \subsection{Jacobi equation}
 If $k\colon [0,\infty) \to (-\infty,0]$ is a smooth function, we denote by $f_k\in C^{\infty}
 \big([0,\infty)\big)$ the solution to the initial value problem
    \begin{equation}\label{jacobieq}
      \left\{ 
	\begin{aligned}
	  f_k'' + kf_k &= 0 \\
	  f_k(0) &= 0, \\
	  f_k'(0) &= 1.
	\end{aligned}\right.
    \end{equation}
 The solution is a non-negative smooth function.   
 
 In later sections we will need some known results related to Jacobi fields and curvature bounds. 
 The proofs of the following three lemmas are based on the Rauch comparison
 theorem (see e.g. \cite{greenewu}) and can be found in \cite{vahakangaspinch}.
 Concerning the curvature bounds, we have the following estimates for the growth of Jacobi fields
 and the Laplacian of the distance function:
 
 \begin{lem}\cite[Lemma 1]{vahakangaspinch}\label{vahakpinclemma1}
    Let $k,K\colon [0,\infty) \to (-\infty,0]$ be smooth functions that are constant in some
    neighborhood of $0$. Suppose that $v\in T_oM$ is a unit vector and $\gamma = \gamma^v
    \colon \R\to M$ is the unit speed geodesic with $\dot{\gamma}_0=v$. Suppose that for every
    $t>0$ we have
      $$
	k(t) \le K_M(P)\le K(t)
      $$
    for every two-dimensional subspace $P\subset T_{\gamma(t)}M$ that contains the radial vector
    $\dot{\gamma}_t$.
      \begin{enumerate}
       \item If $W$ is a Jacobi field along $\gamma$ with $W_0=0,$ $\abs{W_0'}=1$,
	     and $W_0'\bot v$, then
	     $$
	      f_K(t) \le \abs{W(t)} \le f_k(t)
	     $$
	     for every $t\ge0$.
      \item For every $t>0$ we have
	      $$
		(n-1)\frac{f_K'(t)}{f_K(t)} \le \Delta r\big(\gamma(t)\big) \le (n-1)
		\frac{f_k'(t)}{f_k(t)}.
	      $$
       \end{enumerate}

 \end{lem}

 The pinching condition for the sectional curvatures gives a relation between the
 maximal and minimal moduli of Jacobi fields along a given geodesic that contains the radial
 vector:
 \begin{lem}\cite[Lemma 3.2]{cheng}\cite[Lemma 3]{vahakangaspinch}\label{vahakpinclemma3}
    Let $v\in T_oM$ be a unit vector and $\gamma=\gamma^v$. Suppose that $r_0>0$ and $k<0$
    are constants such that $K_M(P)\ge k$ for every two-dimensional subspace $P\subset T_xM$,
    $x\in B(o,r_0)$. Suppose that there exists a constant $C_K<\infty$ such that
      $$
	\abs{K_M(P)}\le C_K\abs{K_M(P')}
      $$
    whenever $t\ge r_0$ and $P,P'\subset T_{\gamma(t)}M$ are two-dimensional subspaces containing
    the radial vector $\dot{\gamma}_t$. Let $V$ and $\bar{V}$ be two Jacobi fields along $\gamma$
    such that $V_0=0=\bar{V}_0$, $V_0'\bot \dot{\gamma}_0 \bot \bar{V}_0$,  and $\abs{V_0'}=1=
    \abs{\bar{V}_0'}$. Then there exists a constant $c_0=c_0(C_K,r_0,k)>0$ such that
      $$
	\abs{V_r}^{C_K} \ge c_0\abs{\bar{V}_r}
      $$
    for every $r\ge r_0$.
 \end{lem}
 
 To prove the solvability of the minimal graph equation, we will need an estimate for the gradient
 of a certain angular function. This estimate can be obtained in terms of Jacobi fields:
 \begin{lem}\cite[Lemma 2]{vahakangaspinch}\label{anglegrad}
    Let $x_0\in M\setminus \{o\}, \, U=M\setminus \gamma^{o,x_0}(\R)$, and define $\theta
    \colon U \to [0,\pi], \, \theta(x) = \an_o(x_0,x)\coloneqq \arccos\ang{\dot{\gamma}_0^{o,x_0},
    \dot{\gamma}_0^{o,x}}$. Let $x\in U$ and $\gamma=\gamma^{o,x}$. Then there exists a 
    Jacobi field $W$ along $\gamma$ with $W(0)=0,\, W_0' \bot \dot{\gamma}_0,$ and $\abs{W_0'}=1$
    such that
      $$
	\abs{\nabla \theta(x)} \le \frac{1}{\abs{W(r(x))}}.
      $$
 \end{lem}

 \subsection{Young functions}
 Let $\phi\colon [0,\infty) \to [0,\infty)$ be a homeomorphism and let $\psi=\phi^{-1}$.
 Define \emph{Young functions} $\Phi$ and $\Psi$ by setting
    $$
      \Phi(t) = \int_0^t \phi(s) \, ds
    $$
 and
    $$
      \Psi(t) = \int_0^t \psi(s) \, ds
    $$
 for each $t\in[0,\infty)$. Then we have the following \emph{Young's inequality}
    $$
      ab\le \Phi(a) + \Psi(b)
    $$
 for all $a,b\in[0,\infty)$. The functions $\Phi$ and $\Psi$ are said to form a \emph{complementary
 Young pair}. Furthermore, $\Phi$ (and similarly $\Psi$) is a continuous, strictly increasing,
 and convex function satisfying
    $$
      \lim_{t\to0^+} \frac{\Phi(t)}{t} = 0
    $$
 and
    $$
      \lim_{t\to\infty} \frac{\Phi(t)}{t} = \infty.
    $$
 For a more general definition of Young functions see e.g. \cite{kufneroldrichfucik}.
 
 As in \cite{vahakangasyoung}, we consider complementary Young pairs of a special type. For that, 
 suppose that a homeomorphism $G\colon [0,\infty) \to [0,\infty)$ is a Young function that 
 is a diffeomorphism on $(0,\infty)$ and satisfies
    \begin{equation}\label{young1}
     \int_0^1 \frac{dt}{G^{-1}(t)} < \infty
    \end{equation}
 and
    \begin{equation}\label{young2}
     \lim_{t\to0} \frac{tG'(t)}{G(t)} = 1.
    \end{equation}
Then we define $F\colon
[0,\infty) \to [0,\infty)$ so that $G$ and $F$ form a complementary
Young pair. The space of such functions $F$ will be denoted by $\F$. Note that if $F \in\F$,
then also $\lambda F\in \F$ and $F(\lambda \cdot) \in \F$ for every $\lambda>0$. In 
\cite{vahakangasyoung} it is proved that for fixed $\ve_0\in(0,1)$ there exists $F\in\F$ 
such that
    \begin{equation}\label{fexists}
      F(t) \le t^{1+\ve_0} \exp\left( -\tfrac{1}{t}\Big(\log\Big(e + \tfrac{1}{t}\Big)
      \Big)^{-1-\ve_0} \right)
    \end{equation}
for all $t\in [0,\infty)$. The construction of such $F$ is done by first choosing $\lambda
\in (1,1+\ve_0)$ and a homeomorphism $H\colon [0,\infty)\to [0,\infty)$ that is a diffeomorphism
on $(0,\infty)$ and satisfies
	  \begin{equation}\label{Hdefin} 
	  H(t) =
	  \begin{cases}
	     \left( \log \tfrac{1}{t} \right)^{-1} \left( \log \log \tfrac{1}{t} 
	     \right)^{-\lambda} & \text{if }t\text{ is small enough,} \\
	     t^{1/\ve_0} & \text{if }t\text{ is large enough},
	  \end{cases}
	  \end{equation}
and then setting $G(t)=\int_0^t H(s) \, ds$ and $F(t)=\int_0^t H^{-1}(s) \, ds$. 
From now on, $G$ and $F$ will denote the complementary
Young pair obtained via this procedure. For details, see \cite{vahakangasyoung} and 
the proof of Proposition \ref{g1exist} below.

Since $G$ is convex, we have $G(t)\ge ct$ for all $t\ge1$. Therefore $G^{-1}(t) \le ct$ for
all $t$ large enough and this implies that $\int_0^{\infty} 1/G^{-1} = \infty$. From this,
together with \eqref{young1}, we conclude that the function $\psi$, defined by
    $$
      \psi(t) = \int_0^t \frac{ds}{G^{-1}(s)},
    $$
is a homeomorphism $[0,\infty)\to[0,\infty)$ that is a diffeomorphism on $(0,\infty)$. 
Hence the same is true for its inverse
    \begin{equation}\label{phidef}
      \varphi = \psi^{-1}\colon [0,\infty) \to [0,\infty).
    \end{equation}
The following lemma collects the properties of $\varphi$.
  \begin{lem}\cite[Lemma 4.5]{vahakangasyoung}
   The function $\varphi  \colon [0,\infty) \to [0,\infty)$ is a homeomorphism that is 
   smooth on $(0,\infty)$ and satisfies
    \begin{equation}\label{phiprop1}
     G\circ\varphi'=\varphi
    \end{equation}
    and
    \begin{equation}\label{phiprop2}
     \lim_{t\to0+} \frac{\varphi''(t)\varphi(t)}{\varphi'(t)^2} = 1.
    \end{equation}
  \end{lem} 
From now on, $\varphi$ will be the function defined in \eqref{phidef} such that the corresponding
$F\in\F$ satisfies \eqref{fexists}. Using the computations done in \cite{vahakangasyoung}, 
we obtain a more specific formula for the
function $\varphi$. Namely, we know that $G^{-1}(t) \approx t/H(t)$ and hence
    $$
      \psi(t) = \int_0^t\frac{ds}{G^{-1}(s)} \approx \int_0^t \frac{1}{s(\log\tfrac{1}{s})
      (\log\log\tfrac{1}{s})^{1+\ve_0}} = \frac{1}{\ve_0}\big(\log\log\tfrac{1}{t}\big)^{-\ve_0}.
    $$
Here and in what follows $\approx$ means that the ratio of the two sides tends to $1$ as
$t\to0^+$.
From this it is straightforward to see that
    \begin{equation}\label{phiexpli}
      \varphi(t) \approx \exp\Big(-\exp \big( \tfrac{1}{\ve_0 t}\big)^{\ve_0} \Big).
    \end{equation}

We will also need complementary Young functions $G_1$ and $F_1$ to deal with the second derivative
of the function $\varphi$. The existence of these 
functions will be proved by the following proposition which is just a modification of 
\cite[Proposition 4.3]{vahakangasyoung} since in the construction of the Young functions we will replace the
function $H$ in \cite{vahakangasyoung} by $H^2$.
   \begin{prop}\label{g1exist}
      Let $\ve_0\in(0,1)$ and $\lambda\in(1,1+\ve_0)$ be as in \eqref{Hdefin}. Then there 
      exist complementary Young functions $G_1$ and $F_1$, 
      and a constant $c>0$ such that $G_1$ satisfies
	\begin{equation}\label{g1prop}
	 G_1\big( \varphi''(t)\big)\approx \varphi(t)
	\end{equation}
       and $F_1$ satisfies 
      \begin{equation}\label{F1estim}
       F_1(t) \le c t \exp \Big( - \tfrac{2^{\lambda}}{\sqrt{t}} \big( \log\tfrac{1}{t} \big)^{-
	      \lambda}  \Big)
      \end{equation}
for all sufficiently small $t>0$.
   \end{prop}
    \begin{proof}
      Let $H\colon [0,\infty) \to [0,\infty)$ be as in \eqref{Hdefin}.
   We define $G_1(t) = \int_0^t H(s)^2 \,ds$. Then $G_1$ is a Young function and we denote by
   $F_1$ its Young conjugate. Notice that $G_1'(t) = H(t)^2$ and that $t(H^2)'(t)/H(t)^2 \to 0$
   as $t\to 0$. Hence, by l'Hospital's rule, we have
	$$
	  \lim_{t\to0} \frac{tG_1'(t)}{G_1(t)} = \lim_{t\to0} \frac{\tfrac{d}{dt}(tG_1'(t))}{G_1'
	  (t)} = 1
	$$
   and we see that $G_1$ satisfies \eqref{young2}. Next, denote $R(t) = t/H(t)^2$. Then it is 
   easy to see that $R(kt) \approx kR(t)$ for every constant $k>0$ and we get
	$$
	  R\big(G_1(t)\big) \approx R\big(tH(t)^2\big) = \frac{tH(t)^2}{H\big(tH(t)^2\big)^2} \approx t,
	$$
   which gives us $G_1^{-1}(t) \approx R(t)$. 
   It follows that $G_1$ satisfies \eqref{young1} and hence $F_1 \in \F$. 
   On the other hand $\varphi(t) = \psi^{-1}(t)$ and 
	$$
	  \psi'(t) = \frac{1}{G^{-1}(t)} \approx \frac{H(t)}{t},
	$$
   and therefore
	$$
	  \varphi'(t) = \frac{1}{\psi'\big(\varphi(t)\big)} \approx \frac{\varphi(t)}{H\big(
	  \varphi(t)\big)}.
	$$
   By \eqref{phiprop2} we obtain 
	$$
	  \varphi''(t) \approx \frac{\varphi(t)}{H\big(\varphi(t)\big)^2} = R\big(\varphi(t)
	  \big) \approx G_1^{-1}\big(\varphi(t)\big),
	$$
   and so  
      \begin{equation}\label{G1phiestim}
      	G_1\big(\varphi''(t)\big) \approx \varphi(t).
      \end{equation}  
 Thus we are left to estimate $F_1$ from above.
   
 It is straightforward to check that 
      $$
	(H^2)^{-1}(t) = \exp\Big(-\exp\big(\lambda W(\lambda^{-1}t^{-1/(2\lambda)})\big)\Big),
      $$
 for all sufficiently small $t$, where $W$ is the Lambert $W$ function defined by the identity
 $W(s)e^{W(s)} = s$. Since $F_1'(t) = (G_1')^{-1}(t) = (H^2)^{-1}(t)$ and
 $W(s) \ge \log s- \log\log s$ for all $s\ge e$, we get for sufficiently small $t$
      \begin{align*}
	F_1(t) &= \int_0^t (H^2)^{-1}(s) \,ds \le t(H^2)^{-1}(t) \\
	&= \frac{t}{\exp\Big(\exp \big( \lambda W(\lambda^{-1}t^{-1/2\lambda}) \big)\Big)} \\
	&\le \frac{t}{\exp\Big(\exp \big( \lambda \log(\lambda^{-1}t^{-1/2\lambda}) -\lambda 
	    \log\log(\lambda^{-1}t^{-1/2\lambda}) \big)\Big)} \\
	&= \frac{t}{\exp\Big( (\lambda^{-1}t^{-1/2\lambda})^{\lambda} \big(\log(
	    \lambda^{-1}t^{-1/2\lambda})\big)^{\lambda}    \Big)} \\
	&= t \exp \Big( -\tfrac{1}{\lambda^{\lambda}\sqrt{t}} \big( \log \tfrac{1}{\lambda}
	     + \tfrac{1}{2\lambda} \log\tfrac{1}{t}  \big)^{-\lambda} \Big) \\
	&\le c t \exp \Big( - \tfrac{2^{\lambda}}{\sqrt{t}} \big( \log\tfrac{1}{t} \big)^{-
	      \lambda}  \Big).
      \end{align*} 	
   \end{proof}

 \subsection{Minimal graph equation}
 Let $\Omega\subset M$ be an open set. Then a function $u\in W^{1,1}_{\loc}(\Omega)$ is a
 \emph{(weak) solution of the minimal graph equation} if
    \begin{equation}
      \int_{\Omega} \frac{\ang{\nabla u,\nabla \varphi}}{\sqrt{1+\abs{\nabla u}^2}} = 0
    \end{equation}
 for every $\varphi\in C_0^{\infty}(\Omega)$. Note that the integral is well-defined since
    $$
      \sqrt{1+\abs{\nabla u}^2} \ge \abs{\nabla u} \quad \text{a.e.},
    $$
 and thus
    $$
      \int_{\Omega} \frac{\abs{\ang{\nabla u,\nabla \varphi}}}{\sqrt{1+\abs{\nabla u}^2}} \le
       \int_{\Omega} \frac{\abs{\nabla u}\abs{\nabla \varphi}}{\sqrt{1+\abs{\nabla u}^2}} \le
       \int_{\Omega} \abs{\nabla \varphi} < \infty.
    $$
    
 It is known that under certain conditions there exists a (strong) solution of \eqref{mingraph}
 with given boundary values. Namely, let $\Omega \subset\subset M$ be a smooth relatively
 compact open set whose boundary has positive mean curvature with respect to inwards
 pointing unit normal. Then for each $f\in C^{2,\alpha}(\bar{\Omega})$ there exists a unique
 $u\in C^{\infty}(\Omega)\cap C^{2,\alpha}(\bar{\Omega})$ that solves the minimal graph equation
 \eqref{mingraph} in $\Omega$ and has the boundary values $u|\p\Omega = f|\p\Omega$.

 \section{Asymptotic Dirichlet problem for minimal graph equation}
 
We begin by the following Caccioppoli-type inequality which will have a crucial role in the  
proof of the solvability of the minimal graph equation.
\begin{lem} \label{caccioppoli}
Suppose $\varphi \colon [0,\infty) \to [0,\infty)$ is a homeomorphism that is smooth on 
$(0,\infty)$ and
let $U\subset\subset M$ be open. Suppose that $\eta\ge0$ is a $C^1(U)$ function and let 
$u,\theta \in L^{\infty}(U)\cap W^{1,2}(U)$ be continuous functions such that 
 $u\in C^2(U)$ is a solution to the minimal graph equation \eqref{mingraph} in $U$.
Denote 
  $$
    h=\frac{\abs{u-\theta}}{\nu},
  $$ 
where $\nu>0$ is a constant, and assume that
  $$
    \eta^2 \varphi(h) \in W^{1,2}_0(U).
  $$
Then we have
  \begin{align} \label{caceq1}
  \int_U \eta^2 \varphi'(h) \frac{\abs{\nabla u}^2}{\sqrt{1+\abs{\nabla u}^2}} 
  &\le C_{\ve} \int_U \eta^2 \varphi'(h) \abs{\nabla \theta}^2 +  (4+\ve)\nu^2 \int_U
  \frac{\varphi^2}{\varphi'}(h) \abs{\nabla \eta}^2
  \end{align}
 for any fixed $\ve>0$.
\end{lem}
  
\begin{proof}
 Define an auxiliary function $f$ by
    $$
      f=\eta^2 \varphi \left( \frac{(u-\theta)^+}{\nu}\right) -\eta^2\varphi  
      \left( \frac{(u-\theta)^-}{\nu}\right).
    $$
 Then it holds that $f\in W_0^{1,2}(U)$ and its gradient is given by
    $$
      \nabla f= \frac{1}{\nu} \eta^2 \varphi'(h)(\nabla u - \nabla \theta) + 2 \eta 
      \,sgn(u-\theta) \varphi(h) \nabla \eta.
    $$
 Since $u$ is a solution to the minimal graph equation, we can use $f$ as a test function in
    $$
      \int_U \frac{\langle \nabla u, \nabla f\rangle}{\sqrt{1+\abs{\nabla u}^2}} = 0,
    $$
 and obtain 
    \begin{align*}
	\int_U \eta^2 \varphi'(h) \frac{\abs{\nabla u}^2}{\sqrt{1+\abs{\nabla u}^2}}
	&= \int_U \eta^2 \varphi'(h) \frac{\langle \nabla u, \nabla \theta \rangle}{\sqrt{1+
		\abs{\nabla u}^2}} \\
	&\qquad - 2\nu \int_U \eta \,sgn(u-\theta) \varphi(h) \frac{\langle \nabla u, \nabla \eta
		\rangle}{\sqrt{1+\abs{\nabla u}^2}} \\
	&\le \int_U \eta^2 \varphi'(h) \frac{\abs{\nabla u} \abs{\nabla \theta}}{\sqrt{1+\abs{\nabla u}^2}} 
		+ 2\nu \int_U \eta \varphi(h) \frac{\abs{\nabla u} \abs{\nabla \eta}}{\sqrt{1+
		\abs{\nabla u}^2}}.
    \end{align*}
Next we use Young's inequality $ab\le (\ve/2)a^2+1/(2\ve)b^2$ and $\sqrt{1+\abs{\nabla u}^2} 
\ge 1$ to estimate the terms on the right hand side as
    $$
      \int_U \eta^2 \varphi'(h) \frac{\abs{\nabla u} \abs{\nabla \theta}}{\sqrt{1+\abs{\nabla u}^2}} 
      \le \frac{\ve_1}{2} \int_U \eta^2 \varphi'(h) \frac{\abs{\nabla u}^2}{\sqrt{1+\abs{\nabla u}^2}} 
      + \frac{1}{2\ve_1} \int_U \eta^2 \varphi'(h) \abs{\nabla \theta}^2
    $$
and
    $$
      2\nu \int_U \eta \varphi(h) \frac{\abs{\nabla u} \abs{\nabla \eta}}{\sqrt{1+\abs{\nabla u}^2}}
      \le \ve_2 \int_U \eta^2 \varphi'(h) \frac{\abs{\nabla u}^2}{\sqrt{1+\abs{\nabla u}^2}}
      + \frac{\nu^2}{\ve_2} \int_U \frac{\varphi^2}{\varphi'}(h) \abs{\nabla \eta}^2.
    $$
Then we choose $\ve_1$ and $\ve_2$ such that $\ve_1$ is small enough and $\ve_2$
minimizes the term
    $$
      \frac{1}{\ve_2(1-\ve_1/2 - \ve_2)}
    $$
i.e. $\ve_2 = (2-\ve_1)/4$. Combining all terms we arrive at
\begin{align*}
      \int_U \eta^2 \varphi'(h) \frac{\abs{\nabla u}^2}{\sqrt{1+\abs{\nabla u}^2}} 
      &\le \frac{2}{\ve_1(2-\ve_1)} \int_U \eta^2 \varphi'(h) \abs{\nabla \theta}^2 +  
      \frac{4 \nu^2}{1-\ve_1} \int_U
      \frac{\varphi^2}{\varphi'}(h) \abs{\nabla \eta}^2 \\
      &= C_{\ve} \int_U \eta^2 \varphi'(h) \abs{\nabla \theta}^2 +  
      (4 +\ve)\nu^2 \int_U
      \frac{\varphi^2}{\varphi'}(h) \abs{\nabla \eta}^2.
    \end{align*}

\end{proof} 

\begin{rem}
    As can be seen in the proof of Lemma \ref{intestimphi}, the second term 
	$$
	  (4+\ve)\nu^2 \int_U \frac{\varphi^2}{\varphi'}(h) \abs{\nabla \eta}^2
	$$
    on the right hand side of \eqref{caceq1} is the only term that affects to the dimension-curvature
    restriction.
\end{rem}

We notice that the left hand side of \eqref{caccioppoli} can be estimated from below by
    \begin{equation}\label{caclow}
      \int_U \eta^2 \varphi'(h) \frac{\abs{\nabla u}^2}{\sqrt{1+\abs{\nabla u}^2}} \ge 
      c_1 \int_{U_1} \eta^2 \varphi'(h) \abs{\nabla u}^2  +
      c_2 \int_{U_2} \eta^2 \varphi'(h) \abs{\nabla u} 
    \end{equation}
where 
    $$
      U_1 = \{\abs{\nabla u} \le \sigma\}, \quad U_2 = \{\abs{\nabla u} \ge \sigma\}, \quad \sigma>0
    $$
and 
    $$
      c_1 = \frac{1}{\sqrt{1+\sigma^2}} 
      , \quad  c_2 = \frac{1}{\sqrt{1+(1/\sigma^2)}}. 
    $$

In the following Lemmas we will obtain some estimates using Lipschitz data $\theta \colon M 
\to \R$. By Rademacher's theorem, Lipschitz functions are differentiable almost everywhere
and throughout the computations, the gradient $\nabla \theta$ appears only inside integrals
so the points where $\theta$ is not differentiable will not be a problem. 

Before stating the 
Lemmas we introduce the following notation. For $x\in M$, we denote  by $j(x)$ the infimum 
of $\abs{V\big(r(x) \big)}$ over Jacobi fields V along the geodesic $\gamma^{o,x}$ that satisfy 
$V_0=0, \,  \abs{V_0'}=1$ and $V_0'\bot \dot{\gamma}_0^{o,x}$. We also note that 
since $M$ is a Cartan-Hadamard manifold, we have
    $$
      \Delta r \ge \frac{n-1}{r}
    $$
  in $M\setminus \{o\}$.
  From the curvature upper bound, Lemma \ref{vahakpinclemma1} and 
  \cite[Example 1]{vahakangaspinch} it follows that for every $\ve>0$ there exists $R_1>R_0$
  such that 
    $$
      \Delta r \ge \frac{(n-1)\phi}{(1+\ve)r}
    $$ 
  for $r \ge R_1$ and therefore
    \begin{equation}\label{laplacerestim}
      r\Delta r \ge \begin{cases}
                     n-1, & \text{in } M\setminus\{o\}, \\
                     \dfrac{(n-1)\phi}{1+\ve}, & \text{in } M\setminus B(o,R_1).
                    \end{cases}
    \end{equation}
  
 \begin{lem}\label{intestimphi}
   Let $M$ be a Cartan-Hadamard manifold satisfying
      $$
	K(P) \le -\frac{\phi(\phi-1)}{r(x)^2},
      $$
   where $K(P)$ is the sectional curvature of any plane $P\subset T_xM$ that contains the
   radial vector field $\nabla r(x)$ and $x$ is any point in $M\setminus B(o,R_0)$. 
   Furthemore, suppose that the dimension of $M$ and the constant $\phi$ satisfies the 
   relation \eqref{dimrestriction}.
   Let $U=B(o,R)$, with $R>R_1,$ 
   and suppose that $u\in C^2(U)\cap C(\bar{U})$ is the unique solution to the minimal graph 
   equation in
   $U$, with $u|\p U= \theta|\p U$, where $\theta\colon M \to \R$ is a Lipschitz function, 
   with $\abs{\nabla\theta(x)} \le 1/j(x)$ almost everywhere. Then there exists a constant $c$ 
   independent of $u$ such that
      $$
	\int_U \varphi(\abs{u-\theta}/c) \le c + c\int_U F(r \abs{\nabla\theta}) + 
	  c\int_U F_1( r^2 \abs{\nabla\theta}^2).
      $$ 
 \end{lem}
  
  \begin{proof}
  As before, we denote $h=\abs{u-\theta}/\nu$, where $\nu\ge\nu_0$ will be fixed later, and to 
  shorten the notation we denote $(n-1)\phi/(1+\ve) \eqqcolon C_0$. 
By splitting the integration domain and using the estimate \eqref{laplacerestim}, we first obtain
    \begin{align*}
     \int_U \varphi(h) r\Delta r &= \int_{B(o,R_1)} \varphi(h) r\Delta r + \int_{U\setminus 
	  B(o,R_1)} \varphi(h) r\Delta r \\
      &\ge (n-1)\int_{B(o,R_1)} \varphi(h) + C_0 \int_{U\setminus B(o,R_1)} \varphi(h) \\
      &\ge (n-1-C_0) \int_{B(o,R_1)} \varphi(h) + C_0\int_U \varphi(h) \\
      &\ge -c + C_0\int_{U} \varphi(h),
    \end{align*}
where $c\ge0$ is some constant. Next we use Green's formula to obtain 
    \begin{align*}
     -c+ C_0\int_U \varphi(h) \le \int_U \varphi(h) r\Delta r &= -\int_U \langle\nabla
      (\varphi(h)r),\nabla r\rangle \\ &= -\int_U \varphi(h) - \int_U r\varphi'(h)\langle 
      \nabla h, \nabla r\rangle,
    \end{align*}
and consequently we have
    $$
      -c+(1+C_0) \int_U \varphi(h) \le \int_U r\varphi'(h) \abs{\nabla h}.
    $$
To estimate the right hand side term, we first split the integration domain into two pieces 
$U=U_1 \cup U_2$, where 
    $$
      U_1 = \{x\in U \colon \abs{\nabla u} \le \sigma\} \quad\text{ and }\quad U_2 = 
      \{x\in U \colon \abs{\nabla u} > \sigma\}.
    $$
Note that $\abs{\nabla h} \le \abs{\nabla u}/\nu + \abs{\nabla \theta}/\nu$, so using the 
Caccioppoli-type inequality \eqref{caceq1} and \eqref{caclow} we get

\begin{align*}
     \int_U r\varphi'(h) \abs{\nabla h} &\le \frac{1}{\nu} \int_{U_1} r \varphi'(h)\abs{\nabla u}
	  + \frac{1}{\nu} \int_{U_2} r\varphi'(h)\abs{\nabla u} + \frac{1}{\nu}\int_{U} r\varphi'(h)\abs{\nabla \theta} \\
     &\le \frac{1}{\nu} \int_{U_1} r \varphi'(h)\abs{\nabla u} + \frac{1}{\nu}\int_{U} r\varphi'(h)\abs{\nabla \theta} \\
	&\qquad  + \frac{C_{\ve}}{c_2\nu} \int_U r\varphi'(h)\abs{\nabla \theta}^2  
      + \frac{(4+\ve)\nu}{c_2} \int_U \frac{\varphi^2}{\varphi'}(h) \abs{\nabla\sqrt{r}}^2 \\
     &= \frac{1}{\nu} \int_{U_1} r \varphi'(h)\abs{\nabla u} + \frac{1}{\nu}\int_{U} r\varphi'(h)\abs{\nabla \theta} \\
	&\qquad  + \frac{C_{\ve}}{c_2\nu} \int_U r\varphi'(h)\abs{\nabla \theta}^2  
      + \frac{(4+\ve)\nu}{4c_2}
	  \int_U \frac{\varphi^2}{\varphi'}(h) r^{-1}  
    \end{align*}
By \eqref{phiprop1} and the convexity of the Young function $G$ we have $\varphi(h) \le c\varphi'(h)$,
and for $r$ large enough, 
$\abs{\nabla\theta}<1$, so $\abs{\nabla\theta}^2 \le \abs{\nabla\theta}$. So from the previous estimate,
we deduce that
      \begin{align*}
	\int_U r\varphi'(h) \abs{\nabla h} &\le \frac{1}{\nu} \int_{U_1} r \varphi'(h)\abs{\nabla 
	u} + \frac{1+C_{\ve}/c_2}{\nu} \int_{U} r\varphi'(h)\abs{\nabla \theta} 
	 + c + \ve' \int_U \varphi(h).
      \end{align*}

We continue again by splitting $U_1$ into two pieces by $U_1 = U_3 \cup U_4,$ where
    $$
      U_3=\left\{\abs{\nabla u} \le \tilde{\sigma} \frac{\varphi(h)}{\varphi'(h)r}\right\} \quad \text{ and } \quad
      U_4=\left\{ \tilde{\sigma} \frac{\varphi(h)}{\varphi'(h)r} < \abs{\nabla u} \le \sigma\right\}
    $$  
and $\tilde{\sigma}$ is a constant to be determined later.
Denote $\Psi(t) \coloneqq \int_0^t \varphi'(s)^2/\varphi(s) \,ds$. 
Then using the Caccioppoli-type inequality \eqref{caceq1} and \eqref{caclow} with $r$ and 
$\Psi'$   instead of $\eta$ and
$\varphi'$ respectively, we can estimate the integral over $U_1$ by
    \begin{align*}
     \int_{U_1} r \varphi'&(h) \abs{\nabla u} \le \tilde{\sigma} \int_{U_3} \varphi(h) + 
     \frac{1}{\tilde{\sigma}} \int_{U_4} r^2 \frac{\varphi'(h)^2}{\varphi(h)} \abs{\nabla u}^2 \\
     &\le \tilde{\sigma} \int_{U_3} \varphi(h) + \frac{1}{\tilde{\sigma}}
     \left( \frac{C_{\ve}}{c_1}\int_U r^2 \Psi'(h)
     \abs{\nabla\theta}^2 + \frac{(4+\ve)\nu^2}{c_1} \int_U \frac{\Psi^2}{\Psi'}(h) \right).
    \end{align*}
From \eqref{phiprop2} we see that 
    $$
     \Psi'(t)= \frac{\varphi'(t)^2}{\varphi(t)} \le \tilde{c} \varphi''(t)
    $$
for $t$ small enough,  and hence
    $$
      \Psi(t) = \int_0^t \frac{\varphi'(s)^2}{\varphi(s)} \le \tilde{c}\varphi'(t),
    $$
which implies that 
    $$
      \frac{\Psi^2}{\Psi'}(h) \le \tilde{c} \frac{\varphi'(h)^2}{\varphi'(h)^2/\varphi(h)} 
      = \tilde{c}\varphi(h).
    $$ 
Notice that $\tilde{c}$, as well as $c_1$, can be chosen arbitrarly close to $1$.
Collecting these estimates together we arrive at 
    \begin{align*}
     (1+C_0)\int_U \varphi(h) &\le c+\ve' \int_U \varphi(h) + \frac{\tilde{\sigma}}{\nu} 
	    \int_U \varphi(h) + \frac{1+C_{\ve}/c_2}{\nu} \int_U r\varphi'(h)\abs{\nabla\theta} \\
	  &\qquad + \frac{C_{\ve}\tilde{c}}{c_1\tilde{\sigma}\nu} \int_U r^2\varphi''(h)\abs{\nabla\theta}^2
		+ \frac{(4+\ve)\nu \tilde{c}}{c_1\tilde{\sigma}}  \int_U \varphi(h). 
    \end{align*}  
Next we use the complementary Young functions $G$ and $F$ to estimate the term with $\varphi'$, 
and $G_1$ and $F_1$ to estimate the term with $\varphi''$. So all together we have 
    \begin{align*}
     \bigg(1+ C_0 - \ve' -\frac{1+C_{\ve}/c_2}{\nu} &- \frac{\tilde{\sigma}}{\nu} - 
     \frac{C_{\ve}\tilde{c}}{c_1\tilde{\sigma}\nu} - \frac{(4+\ve)\nu \tilde{c}}{c_1\tilde{\sigma}}\bigg) \int_U \varphi(h) 
     \\ &\le c+ \frac{1+C_{\ve}/c_2}{\nu} \int_U F(r\abs{\nabla \theta}) + 
	  \frac{C_{\ve}}{c_1\tilde{\sigma}\nu} \int_U F_1(r^2 \abs{\nabla\theta}^2).
    \end{align*} 
For any fixed $\tilde{\ve}>0$,  we can choose
first $\sigma$ and $\ve$ small enough, then $\nu$ big enough and $\tilde{\sigma} 
= \nu$ such that the coefficient on the left hand side is positive provided that
$C_0>4+\tilde{\ve}$. This last inequality is satisfied thanks to the dimension-curvature 
restriction \eqref{dimrestriction} and
hence the claim is proved.
  \end{proof}

 The next lemma is a modification of \cite[Lemma 20]{casterasholopainenripoll2} (or originally 
 \cite[Lemma 2.20]{vahakangasyoung}). The proof is based on the idea of Moser iteration
 procedure.
 
 \begin{lem}\label{moserite}
 
 Let $\Omega=B(o,R)$ and suppose that $\theta\colon \Omega \to \R$ is a bounded Lipschitz 
 function with $\abs{\theta}, \abs{\nabla \theta}
 \le C_1$. Let $u\in C^2(\Omega)$ be a solution of the minimal graph equation in $\Omega$ 
 such that $u$ has the boundary values $\theta$ and $\inf_{\Omega} \theta \le u \le 
 \sup_{\Omega} \theta$. Fix $s\in (0,r_S)$, where 
 $r_S$ is the radius of the Sobolev inequality \eqref{sobieq}, and suppose that $B=B(x,s)
 \subset \Omega$. Then there exists a positive constant $\nu_0 = \nu_0(\varphi,C_1)$
 such that for all fixed $\nu\ge\nu_0$
    $$
      \sup_{B(x,s/2)} \varphi \big(\abs{u-\theta}/\nu\big)^{n+1} \le c\int_B \varphi 
      \big(\abs{u-\theta}/\nu\big),
    $$ 
 where $c$ is a positive constant depending only on $n,\nu,s,C_S,C_1$ and $\varphi$.
 \end{lem}

\begin{rem}
Before proving the Lemma we note that increasing the constant $\nu$ above increases also the
constant $c$. However, it does not cause problems since $\nu$ will always be a fixed constant.
\end{rem}

 \begin{proof}[Proof of Lemma \ref{moserite}]
  We denote $\kappa = n/(n-1), \, B/2 = B(x, s/2)$, and $h=\abs{u-\theta}/\nu$, where 
  $\nu \ge \nu_0>0$ will be fixed later. For each $j \in\N$ we denote $s_j = s(1+
  \kappa^{-j})/2$ and $B_j=B(x,s_j)$. Note that $s_j \to s/2$ as $j\to\infty$. Let $\eta_j$
  be a Lipschitz function such that $0\le \eta_j \le 1, \, \eta_j | B_{j+1}\equiv 1, \, 
  \eta_j | (M\setminus B_j) \equiv 0,$ and that
      $$
	\abs{\nabla \eta_j} \le \frac{1}{s_j-s_{j+1}} = 2n\kappa^j /s.
      $$
  For every $m\ge1$, we have
      $$
	\abs{\nabla \eta_j^2 \varphi(h)^m} \le 2\eta_j \varphi(h)^m \abs{\nabla \eta_j} + 
	  m\eta_j^2 \varphi'(h) \varphi^{m-1}(h)\abs{\nabla h}.
      $$
  First we claim that
      \begin{equation}\label{inta1}
       \left(\int_{B_{j+1}} \varphi(h)^{\kappa m}  \right)^{1/\kappa} \le c(\kappa^j +m+
       \kappa^{2j}/m) \int_{B_j} \varphi^{m-1}.
      \end{equation}
  We notice that, for every $m,j\ge1, \, \eta_j^2\varphi(h)^m$ is a Lipschitz function
  supported in $B_j$. Using the Sobolev inequality \eqref{sobieq}, we first have
      \begin{align}\label{inta2}
       &\left(\int_{B_{j+1}} \varphi(h)^{\kappa m}  \right)^{1/\kappa} \le 
	  \left(\int_{B_{j}} \big(\eta_j^2\varphi(h)^m\big)^{\kappa}  \right)^{1/\kappa} \le
	  C_S \int_ {B_j} \abs{\nabla \big(\eta_j \varphi(h)^m\big)} \nonumber \\
       &\qquad \le 2C_S \int_ {B_j} \eta_j \varphi(h)^m\abs{\nabla \eta_j} + C_S \int_{B_j} \eta_j^2
	  (\varphi^m)'(h) \abs{\nabla h} \nonumber \\
       &\qquad \le c\kappa^j \int_ {B_j} \varphi(h)^m + \frac{C_S}{\nu} \int_{B_j}(\varphi^m)'(h)
	  \abs{\nabla \theta} \\ 
       &\qquad \quad + \frac{C_S}{\nu} \int_{B_j} \eta_j^2 (\varphi^m)'(h)\abs{\nabla u}. \nonumber
      \end{align}
  From the assumption 
      $$
	-C_1 \le \inf_{\Omega} \theta \le u \le \sup_{\Omega} \theta \le C_1
      $$
  we obtain that $\abs{u-\theta}\le 2C_1$. We can use this to obtain upper bounds for $\varphi$ 
  and $\varphi'$. Namely, we have $G\circ \varphi' = \varphi$, where $G\colon [0,\infty) \to
  [0,\infty)$ is the homeomorphic and convex Young function. Consequently there exist constants
  $\nu_0$ and $c$ such that
      $$
	\varphi(h)\le 1, \, \varphi'(h) \le 1\, \text{ and }\, \varphi(h) \le c \varphi'(h)
      $$
  whenever $\nu\ge\nu_0$. Thus we get estimates
      \begin{align}\label{inta3}
	\int_{B_j} \varphi(h)^m \le \int_{B_j} \varphi(h)^{m-1}
      \end{align}
 and
      \begin{align}\label{inta4}
	\int_{B_j} (\varphi^m)'(h)\abs{\nabla\theta} = m\int_{B_j} \varphi(h)^{m-1}\varphi'(h)
	\abs{\nabla\theta} \le mC_1\int_{B_j}\varphi(h)^{m-1}.
      \end{align}
 The third term on the right hand side of \eqref{inta2} can be estimated first as
      \begin{align}\label{inta5}
	\int_{B_j} \eta_j^2(\varphi^m)'(h)\abs{\nabla u} &\le \int_{B_j \cap U_1} \eta_j^2
	    (\varphi^m)'(h) + \int_{B_j \cap U_2} \eta_j^2 (\varphi^m)'(h) \abs{\nabla u} \nonumber \\
	&\le \int_{B_j} m\varphi(h)^{m-1} + \sqrt{2} \int_{B_j} \eta_j^2 (\varphi^m)'(h)
	    \frac{\abs{\nabla u}^2}{\sqrt{1+\abs{\nabla u}^2}},
      \end{align}
 where $U_1$ is the set where $\abs{\nabla u}<1$ and $U_2$ the set where $\abs{\nabla u}\ge1$. The
 constant $\sqrt{2}$ comes from \eqref{caclow} when we choose $\sigma=1$.
 
 Next we notice that $\eta_j^2\varphi(h)^m \in W_0^{1,2}(B_j)$, since $\text{supp }\eta_j \subset
 \bar{B}_j$, and thus we can apply the Caccioppoli-type inequality \eqref{caceq1} with $\varphi^m$
 instead of $\varphi$. We also choose $\ve_1=\ve_2 =1/3$ in the proof of \eqref{caceq1} so
 the constants become $3$ and $6$. Hence we obtain
      \begin{align}\label{inta6}
	\sqrt{2} \int_{B_j} \eta_j^2 (\varphi^m)'(h) &\frac{\abs{\nabla u}^2}{\sqrt{1+\abs{\nabla u}^2}}
	\le 3\sqrt{2} \int_{B_j} \eta_j^2 (\varphi^m)'(h)\abs{\nabla \theta}^2\nonumber \\ &\qquad + 6\sqrt{2}\nu^2
	\int_{B_j} \frac{\varphi^{2m}}{(\varphi^m)'}(h)\abs{\nabla \eta_j}^2 \nonumber \\
	&\le c(m+\kappa^{2j}/m) \int_{B_j} \varphi(h)^{m-1}.
      \end{align} 
 Now the estimate \eqref{inta1} follows by inserting the estimates \eqref{inta3}-\eqref{inta6}
 into \eqref{inta2}. We apply \eqref{inta1} with $m=m_j + 1$, where $m_j=(n+1)\kappa^j -n$.
 Note that $m_{j+1} = \kappa(m_j +1)$, so we can write \eqref{inta1} as
      $$
	\left( \int_{B_{j+1}} \varphi(h)^{m_j} \right)^{1/\kappa} \le C\kappa^j \int_ {B_j}
	\varphi(h)^{m_j}.
      $$
 By denoting 
      $$
	I_j = \left( \int_{B_{j}} \varphi(h)^{m_j} \right)^{1/\kappa^j}
      $$
 we can write the previous inequality as a recursion formula
      $$
	I_{j+1} \le C^{1/\kappa^j} \kappa^{j/\kappa^j} I_j.
      $$
 Since 
      $$
	\limsup_{j\to\infty} I_j \ge \lim_{j\to\infty} \left( \int_{B/2} \varphi(h)^{m_j} 
	\right)^{(n+1)/m_j} = \sup_{B/2} \varphi(h)^{n+1},
      $$ 
 we get
      $$
	\sup_{B/2} \varphi(h)^{n+1} \le \limsup_{j\to\infty} I_j \le C^n \kappa^S I_0 
	\le c \int_{B} \varphi (h),
      $$
 where
      $$
	S = \sum_{j=0}^{\infty} j \kappa^{-j} < \infty.
      $$
  
 \end{proof}

 In order to prove that our solution to the minimal graph equation extends to the boundary
 $\pinf M$ and has the desired boundary values, we will also need that the right hand side integrals
 of Lemma \ref{intestimphi} are finite. The following ensures that the functions $F$ and 
 $F_1$ decrease fast enough. Recall that $j(x)$ denotes the infimum of $\abs{V\big(r(x)
 \big)}$ over Jacobi fields V along the geodesic $\gamma^{o,x}$ that satisfy $V_0=0, \, 
 \abs{V_0'}=1$ and $V_0'\bot \dot{\gamma}_0^{o,x}$.
 \begin{lem}\label{fsmall}
  Let $M$ be a Cartan-Hadamard manifold satisfying
      $$
	K(P) \le -\frac{\phi(\phi-1)}{r(x)^2},
      $$
   where $K(P)$ is the sectional curvature of any plane $P\subset T_xM$ that contains the
   radial vector field $\nabla r(x)$ and $x$ is any point in $M\setminus B(o,R_0)$. 
   Then there exist $F,F_1 \in \F$ such that
      $$
	F\left(\frac{r(x)}{j(x)}\right)j(x)^{C(n-1)} \le r(x)^{-2}
      $$
  and
      $$
        F_1\left(\frac{r(x)^2}{j(x)^2}\right)j(x)^{C(n-1)} \le r(x)^{-2}
      $$
  for any positive constant $C$ and for every $x\in M$ outside a compact set.
  \end{lem}

  \begin{proof} 
   We prove the claim only for function $F$ since the case with $F_1$ (given by Proposition
   \ref{g1exist}) is essentially the same.
   Let $\lambda$ be as in Proposition \ref{g1exist}. By \eqref{fexists} there exists $F\in 
   \F$ such that
      $$
	F(t) \le \exp\left( -\tfrac{1}{t}\Big(\log\Big(e + \tfrac{1}{t}\Big)
      \Big)^{-\lambda} \right)
      $$
   for all small $t$. Hence the claim follows if
      $$
      \exp \left(-\tfrac{j(x)}{r(x)} \Big(\log\Big( e+\tfrac{j(x)}{r(x)}\Big)\Big)^{-\lambda}\right)
      j(x)^{C(n-1)} \le r(x)^{-2},
      $$ 
   and taking logarithms, we see that this is equivalent with
      $$
	\frac{j(x)}{r(x)} \left(\log\Big(e+\tfrac{j(x)}{r(x)}\Big) \right)^{-\lambda} - 
	C(n-1)\log j(x) - 2\log r(x) \ge 0.
      $$
  It follows from the curvature assumptions that $j(x) \ge cr(x)^{\phi}, \, \phi>1,$ whenever
  $r(x)\ge\tilde{R}$ for some $\tilde{R}>0$ (see e.g. Lemma \ref{vahakpinclemma1} and 
  \cite[Example 1]{vahakangaspinch}), so it is enough to show that
      $$
	f(t) \coloneqq \frac{t}{a}\Big(\log \big( e + \tfrac{t}{a} \big)\Big)^{-\lambda}
	    - C(n-1)\log t - 2\log a \ge 0
      $$
  for all $t\ge ca^{\phi}$ when $a$ is big enough. A straightforward computation gives that
      $$
	f'(t) = \frac{\tfrac{1}{a}\left(1- \tfrac{\lambda}{\log(e+ t/a) (ae/t +1)} 
	\right)}{\Big(\log \big( e + \tfrac{t}{a} \big)\Big)^{\lambda}} 
	    - \frac{C(n-1)}{t},
      $$
  so noticing that $t/a \ge ca^{\phi-1} \ge \tilde{R}^{\phi}$ and $\log(e + t/a) \le 
  k(t/a)^{\alpha}$, where $k$ is a constant and $\alpha>0$ can be made as small as we wish,
  we obtain
      $$
	f'(t) \ge \frac{k}{a^{1-\alpha} t^{\alpha}} - \frac{C(n-1)}{t} \ge 0
      $$
  for all $t\ge ca^{\phi}$ and $a$ large enough. Finally we notice that
      \begin{align*}
	f(a^{\phi}) &= a^{\phi-1} \big( \log(e+a^{\phi-1}) \big)^{-\lambda} -C(n-1)
		\log a^{\phi-1} - 2\log a \\
	  &= a^{\phi-1} \big( \log(e+a^{\phi-1}) \big)^{-\lambda} -\big(C(n-1)(\phi-1) + 2\big)
		\log a
      \end{align*}
  which clearly is positive when $a\ge \tilde{R}$ is large enough.
 
  \end{proof}

 \subsection{Solving the asymptotic Dirichlet problem with Lipschitz boundary data}
 
 In order to prove the main theorem we begin by solving the corresponding Dirichlet problem
 with Lipschitz boundary data.
 The asymptotic boundary $\pinf M$ is homeomorphic to the unit sphere $\Ss^{n-1} \subset T_oM$ and
 hence we may interpret the given boundary function $f\in C(\pinf M)$ as a continuous function
 on $\Ss^{n-1}$. We first solve the asymptotic Dirichlet problem for \eqref{mingraph} with
 Lipschitz continuous boundary values $f\in C(\Ss^{n-1})$. We assume that, for all $x\in M$
 and for all $2$-planes $P\subset T_xM$,
    \begin{equation}
     K(P) \le -a^2\big(r(x)\big),
    \end{equation}
where $a\colon[0,\infty)\to[0,\infty)$ is a smooth function that is constant in some 
neighborhood of $0$ and
     $$
      a^2(t) = \frac{\phi(\phi-1)}{t^2}, \quad \phi>1,
     $$
for $t\ge R_0$. Identify $\pinf M$ with the unit sphere $\Ss^{n-1}\subset T_oM$ and assume 
that $f\colon \Ss^{n-1} \to \R$ is $L$-Lipschitz. We extend $f$ radially to a continuous 
function $\theta$ on $M\setminus \{o\}$. The radial extension $\theta$ is also a locally Lipschitz
function and hence, by Rademacher's theorem, differentiable almost everywhere. The gradient
of $\theta$ can be estimated in terms of an angle function as follows. Let $x,y \in \bM$ and
let $\gamma^{o,x}$ and $\gamma^{o,y}$ be the unique unit speed geodesics joining $o$ to $x$ 
and $y$. Denote by $\bar{x}$ and $\bar{y}$ the corresponding points on $\Ss^{n-1}$ i.e.
$\bar{x} = \dot{\gamma}^{o,x}_0$ and $\bar{y} = \dot{\gamma}^{o,y}_0$. Then 
    \begin{align*}
     \frac{\abs{\theta(x)- \theta(y)}}{d(x,y)} &= \frac{\abs{\theta(\bar{x})- \theta(\bar{y})}}{d(x,y)} 
     \le \frac{Ld(\bar{x},\bar{y})}{d(x,y)}\\ &= L \frac{\an_o(\bar{x},\bar{y})}{d(x,y)} 
     =L \frac{\an_o(x,y)}{d(x,y)}
    \end{align*}
and we obtain $\abs{\nabla\theta}\le L\abs{\nabla \an_o(\cdot,\cdot)}$. By Lemma \ref{anglegrad} this
implies
    $$
      \abs{\nabla \theta (x)} \le \frac{L}{j(x)}
    $$
and we see that $\theta$ satisfies the assumptions of Lemmas \ref{intestimphi} and 
\ref{moserite}.

We are now ready to solve the asymptotic Dirichlet problem with Lipschitz boundary data.

\begin{lem}\label{lipsol}
 Let $M$ be a Cartan-Hadamard manifold of dimension $n\ge 2$ satisfying the curvature
 assumptions \eqref{curvassump}, \eqref{pinchassump} and \eqref{dimrestriction} for all 
 2-planes $P\subset T_xM$ with $x\in M \setminus B(o,R_0)$.
 Suppose that $f\in C(\pinf M)$ is $L$-Lipschitz when interpreted as a function on $\Ss^{n-1}
 \subset T_oM$. Then the asymptotic Dirichlet problem for minimal graph equation 
 \eqref{mingraph} is uniquely solvable with boundary data $f$.
\end{lem}

\begin{proof}
 Let $\theta$ be the radial extension of the given Lipschitz boundary data $f\in C(\pinf M)$
 defined above. We exhaust $M$ by an increasing sequence of geodesic balls $B_k=B(o,k),\, 
 k\in\N$, and show first that there exist smooth solutions $u_k\in C^{\infty}(B_k)\cap 
 C(\bar{B}_k)$ of the minimal graph equation
    \begin{align}\label{mingraphbk}
      \left\{ \begin{array}{ll} 
	  \dv \dfrac{\nabla u_k}{\sqrt{1+\abs{\nabla u_k}^2}} = 0, & \text{in } B_k, \\
	  u_k|\p B_k = \theta|\p B_k. 
      \end{array} \right.
    \end{align}
 For this, fix $k\in\N$ and let $(\theta_i^k) \subset C^2(\p B_k)$ be a sequence that converges
 uniformly to the function $\theta$ on $\p B_k$. For every $i$ there exists a function $u_i^k
 \in C^{\infty}(B_k)$ that solves the minimal graph equation in $B_k$ and has boundary values
 $\theta_i^k$. By the Maximum principle we have
    $$
      \sup_{B_k}\abs{u_j^k - u_i^k} \le \sup_{\p B_k} \abs{\theta_j^k - \theta_i^k}
    $$
 so the sequence $(u_i^k)$ converges uniformly to some function $u_k\in C(\bar{B}_k)$. 
 In $\bar{B}_k$ the sectional curvatures are bounded, so we can apply the
 interior gradient estimate \cite[Theorem 1.1]{spruck} and obtain that $\abs{\nabla u_i^k}$
 is locally bounded independent of $i$. Therefore standard arguments and 
 regularity theory of elliptic PDEs imply that $u_i^k \to u_k$ in $C^2_{\loc}(B_k)
 \cap C(\bar{B}_k)$ and therefore $u_k$ is also a solution to the minimal graph equation
 \eqref{mingraphbk}. Moreover, the comparison principle 
 implies that
    $$
      -\max_{x\in M} \abs{\theta(x)} \le u_k \le \max_{x\in M} \abs{\theta(x)},
    $$
 so the solutions $u_k$ are bounded in $B_k$ for every $k\in\N$.
 
 Fix a compact set $K\subset M$. Then applying the interior gradient estimate 
 \cite[Theorem 1.1]{spruck}, we obtain 
    $$
      \sup_K \abs{\nabla u_k} \le c(K),
    $$
 where the constant $c(K)$ is independent of $k$. The theory of elliptic PDEs implies that
 there exists a subsequence, still denoted by $u_k$, that converges in $C^2_{\loc}(M)$
 to a solution $u\in C^{\infty}(M)$. Hence we are left to prove that $u$ extends continuously
 to the boundary $\pinf M$ and satisfies $u|\pinf M = f$.
 
 Next we will use Lemma \ref{intestimphi}, and in order to estimate the appearing integrals we 
 use geodesic polar
 coordinates $(r,v)$ for points $x\in M$. Here we denoted $r=r(x)$ and $v=\dot{\gamma}_0^{o,x}\in S_oM$.
 Let $\lambda(r,v)$ be the Jacobian for these polar coordinates. Note that then we have 
 $\lambda(r,v) \le J(r,v)^{n-1}$ where $J(x)$ denotes the supremum of $\abs{V\big(r(x)
 \big)}$ over Jacobi fields V along the geodesic $\gamma^{o,x}$ that satisfy $V_0=0, \, 
 \abs{V_0'}=1$ and $V_0'\bot \dot{\gamma}_0^{o,x}$.

 Let $\nu$ be such that it satisfies the assuptions of Lemmas \ref{intestimphi} and 
 \ref{moserite}. Applying Lemma \ref{vahakpinclemma3}, Fatou's lemma, and Lemma \ref{intestimphi} with $U=B_k$ we get 
    \begin{align}\label{useofF}
     \int_M &\varphi\big(\abs{u-\theta}/\nu\big) \le \liminf_{k\to\infty} \int_{B_k} \varphi\big(
     \abs{u_k-\theta}/\nu\big)\nonumber \\
     &\le c + c\int_{M} F( r \abs{\nabla\theta}) + 
	  c\int_{M} F_1( r^2 \abs{\nabla\theta}^2) \nonumber \\
     &= c + c\int_{R_1}^{\infty} \int_{S_oM} F( r \abs{\nabla\theta(r,v)})
	      \lambda(r,v)\, dv\,dr \nonumber \\ 
     &\qquad + c\int_{R_1}^{\infty} \int_{S_oM} F_1( r^2 \abs{\nabla\theta(r,v)}^2)
	      \lambda(r,v)\, dv\,dr \nonumber  \\
     &\le c + c\int_{R_1}^{\infty} \int_{S_oM} F\left( \frac{r}{j(r,v)}\right)
	      j(r,v)^{C_K(n-1)}\, dv\,dr  \nonumber \\ 
     &\qquad + c\int_{R_1}^{\infty} \int_{S_oM} F_1\left( \frac{r^2}{j(r,v)^2}\right)
	      j(r,v)^{C_K(n-1)}\, dv\,dr \nonumber \\
	  &< \infty.
    \end{align}
Finiteness of the last integrals follows from Lemma \ref{fsmall}.

 Let $x\in M$ and fix $s\in (0,r_S)$. For $k$ large enough, $u_k$ satisfies the assumptions
 of Lemma \ref{moserite}, and hence
    $$
      \sup_{B(x,s/2)} \varphi\big(\abs{u_k-\theta}/\nu\big)^{n+1} \le c\int_{B(x,s)}\varphi\big(
      \abs{u_k-\theta}/\nu\big).
    $$
 This and the dominated convergence theorem implies that 
    \begin{align}\label{dklestim}
     \sup_{B(x,s/2)} &\varphi\big(\abs{u-\theta}/\nu\big)^{n+1} = \sup_{B(x,s/2)} \lim_{k\to\infty} 
	  \varphi\big(\abs{u_k-\theta}/\nu\big)^{n+1}\nonumber \\
     &\le \limsup_{k\to\infty} \sup_{B(x,s/2)} \varphi\big(\abs{u_k-\theta}/\nu\big)^{n+1} \\
     &\le c \limsup_{k\to\infty} \int_{B(x,s)} \varphi\big(\abs{u_k-\theta}/\nu\big)
     = c \int_{B(x,s)} \varphi\big(\abs{u-\theta}/\nu\big). \nonumber
    \end{align}
Let $\xi \in \pinf M$ and $(x_i)$ be a sequence of points in $M$ with $x_i \to \xi$ as $i\to
\infty$. Applying the estimate \eqref{dklestim} with $x=x_i$ and fixed $s\in(0,r_S)$ we obtain,
by \eqref{useofF}, that 
    $$
      \lim_{i\to\infty} \sup_{B(x_i,s/2)} \varphi\big(\abs{u-\theta}/\nu\big)^{n+1} \le
	c \lim_{i\to\infty} \int_{B(x_i,s)}\varphi\big(\abs{u-\theta}/\nu\big) =0
    $$
 and hence $\abs{u(x_i)-\theta(x_i)} \to 0$ as $i\to \infty$. Since $\xi\in\pinf M$ was arbitrary,
 it follows that $u$ extends continuously to $\pinf M$ and satisfies $u|\pinf M = f$.

\end{proof}

\subsection{Proof of the main theorem}
Let $f\in C(\pinf M)$. As in the case of Lipschitz functions, we identify $\pinf M$ with
the unit sphere $\Ss^{n-1}\subset T_oM$. Let $(f_i)$ be a sequence of Lipschitz functions
such that $f_i \to f$ uniformly as $i\to\infty$. By Lemma \ref{lipsol} there exist solutions
$u_i \in C^{\infty}(M)\cap C(\bM)$ of the minimal graph equation \eqref{mingraph} with
the desired boundary values $u_i = f_i$ on $\pinf M$. It follows from the Maximum principle
that 
    $$
      \sup_M \abs{u_i-u_j} = \max_{\pinf M} \abs{f_i-f_j}
    $$
and consequently the sequence $u_i$ converges uniformly to a function $u\in C({\bM})$. 
Applying the interior gradient estimate \cite[Theorem 1.1]{spruck} in compact subsets of $M$
we conclude that the convergence takes place in $C(\bM)\cap C^2_{\loc}(M)$ and 
therefore $u$ is also a solution to \eqref{mingraph} in $M$ and $u=f$ on $\pinf M$. 
Regularity theory implies that $u\in C^{\infty}(M)$.

For the proof of uniqueness, suppose that $u$ and $v$ are both solutions of \eqref{mingraph}
in $M$, continuous in $\bM$ and $u=v$ on the boundary $\pinf M$. By symmetry we can
assume that $u(y)>v(y)$ for some $y\in M$. Denote $\delta = \big(u(y) - v(y)\big)/2$ and
let $U \subset \{x\in M \colon u(x) > v(x) +\delta\}$ be the component that contains $y$.
Then $U$ is a relatively compact open domain since both $u$ and $v$ are continuous and 
coincide on $\pinf M$. Furthemore $u=v+\delta$ on $\p U$ and it follows that $u=v+
\delta$ in $U$ which is a contradiction since we have $y\in U$. 

\qed

\bibliographystyle{plain}

\end{document}